\documentclass{amsart}

\usepackage{amsrefs}

\newtheorem{theorem}{Theorem}[section]

\theoremstyle{definition}

\theoremstyle{remark}

\numberwithin{equation}{section}

\newcommand{\grad}{\mathop{\rm grad}\nolimits}
\renewcommand{\div}{\mathop{\rm div}\nolimits}

\begin{document}

\title[Additive schemes for some systems of evolutionary equations]
{Additive schemes (splitting schemes) for some systems of evolutionary equations}

\author{Petr N. Vabishchevich}
\address{Nuclear Safety Institute,
Russian Academy of Sciences,
52, B. Tulskaya, 115191 Moscow, Russia}
\email{vab@ibrae.ac.ru}


\date{}


\begin{abstract}
On the basis of additive schemes (splitting schemes) we construct efficient numerical 
algorithms to solve approximately the initial-boundary value problems 
for systems of time-dependent partial differential equations (PDEs).  
In many applied problems the individual components of the vector of unknowns are
coupled together and then splitting schemes are applied in order to get a simple problem 
for evaluating components at a new time level. 
Typically, the additive operator-difference schemes for systems 
of evolutionary equations are constructed for operators coupled in space. 
In this paper we investigate more general problems where coupling of 
derivatives in time for components of the solution vector takes place. 
Splitting schemes are developed using an additive representation for both
the primary operator of the problem and the operator at the time derivative.  
Splitting schemes are based on a triangular two-component representation of the operators. 
\end{abstract}

\maketitle

\section*{Introduction}

Mathematical modeling of applied problems is based on approximate solving 
boundary value problems for systems of time-dependent PDEs. 
To construct numerical algorithms for solving such problems, we develop 
approximations to the equations taking into account the corresponding 
initial and boundary conditions. The approximation in space is conducted using finite difference schemes, 
finite element procedures or finite volume methods  \cite{Samarskii,Angermann,Saleri}.
Special attention should be given to approximations in time 
for numerical solving problems with systems of equations  \cite{Gustafsson,Ascher,LeVeque}.
In addition to the general conditions of approximation and stability we should keep 
in mind the questions of computational implementation of the constructed schemes. 
To solve the corresponding grid problem at the new time level in an efficient way, special additive 
operator-difference schemes (splitting schemes) are in common use  \cite{Yanenko,Marchuk,Samarskii}. 

In approximate solving initial-boundary value problems for multidimensional PDEs, 
a transition to a chain of simpler problems allows us 
to construct economical difference schemes --- we have splitting with respect to spatial variables.  
In some cases it is useful to split into the subproblems of different nature --- 
we have splitting into physical processes  \cite{Marchuk,SamVabAdditive}. 
Recently, active discussions are concerned with regionally additive schemes 
(domain decomposition schemes), which are oriented to parallel computers \cite{Mathew,SamMatVab}. 

The additive schemes for vector problems can be treated as a separate class. 
The schemes of this type can be used to develop efficient numerical algorithms 
for solving time-dependent systems of PDEs. 
A typical situation is the case where the individual components of the unknown  vector are
interconnected and it is difficult to derive a simple problem for evaluating vector components at the new time level. 

Various classes of additive schemes are developed for vector problems (see, e.g. \cite{SamVabVect}).
For parabolic and hyperbolic systems of equations with a self-adjoint elliptic operator 
the locally one-dimensional additive schemes are constructed in  \cite{Samarskii}
using the principle of regularization for difference schemes. 
To construct efficient splitting schemes for the systems of equations, 
the Samarskii alternating triangular method can be employed, which is generally treated as 
an iterative method \cite{Samarskii,SamNikBook}.
This approach is implemented, in particular, in \cite{VabLis}
for the dynamic problems of elasticity, and in \cite{SamVabNS} --- 
for the problems of an incompressible fluid with a variable viscosity. 
Additive schemes for time-dependent vector equations of first and second order
are constructed in  \cite{VabVecPap} for problems of electrodynamics.

Traditionally, additive operator-difference schemes for systems of evolutionary equations 
are in common use to construct approximations to equations with operators interconnected in space. 
But in some cases there does exist a coupling between the time derivatives of the individual components of 
the solution vector. Therefore, it is necessary to design  additive 
operator-difference schemes with splitting an operator at the time derivative. 
Theory and practice of constructing the splitting schemes of such kind is right at the beginning of the development. 
In fact, the first work on the splitting schemes for the problems with an additive 
representation of an operator at the time derivative for the evolutionary
equations of first order is the paper \cite{VabBadd}. 
New vector additive schemes are proposed and investigated in it using splitting the operator
at the time derivative into a sum of positive definite self-adjoint operators. 
Unfortunately, these schemes cannot be directly applied to the systems of evolutionary 
equations with a coupling between the time derivatives. 

In this paper we construct splitting schemes using an additive representation 
for both the primary operator of the problem and the operator at the time derivative.
The schemes are based on a triangular two-component  representation of the operators and 
are applied to systems of evolutionary equations. The work is organized as follows. 
In Section 1 we formulate the Cauchy problem for a system of evolutionary equations of first order. 
We also provide an example of the equation system with PDEs, 
which can be related to the considered Cauchy problem with the corresponding approximations in space. 
The standard two-level operator-difference scheme is discussed in Section 2. 
Section 3 deals with the construction of additive schemes via the triangular 
splitting of the primary operator of the problem. 
The general problem with splitting both the primary operator of the problem and the operator at the time derivative
is presented in Section 4.  

\section{Problem formulation}

Let $H_{\alpha}, \ \alpha =1,2,\dots,p$ be finite-dimensional real Hilbert (Euclidean) 
spaces where the scalar product and  the norm are $(\cdot,\cdot)_{\alpha}$ and 
$\|\cdot\|_{\alpha}, \ \alpha =1,2,\dots,p$, respectively.
The individual components of the solution are denoted by $u_{\alpha}(t), \ \alpha =1,2,\dots,p$ 
for any $t$ ($0 \leq t \leq T, \ T > 0$). 
We search the solution for the system of evolutionary equations of first order: 
\begin{equation}\label{1}
  \sum_{\beta =1}^{p} B_{\alpha \beta} \frac{d u_{\beta}}{d t} +
  \sum_{\beta =1}^{p} A_{\alpha \beta}  u_{\beta} = f_{\alpha},
  \quad \alpha = 1, 2, \dots, p. 
\end{equation}
Here $f_{\alpha}(t) \in L_2(0,T; H_{\alpha}), \ \alpha =1,2,\dots,p$ are specified, and 
$B_{\alpha \beta}, \ A_{\alpha \beta}$ are the linear constant (independent of $t$) operators
acting from $H_{\beta}$ into $H_{\alpha}$ 
($A_{\alpha \beta}: H_{\beta}\to H_{\alpha}$, 
$B_{\alpha \beta}: H_{\beta}\to H_{\alpha}$)
for all $\alpha =1,2,\dots,p$.
The system of equations (\ref{1}) is supplemented with the initial conditions 
\begin{equation}\label{2}
  u_{\alpha}(0) = v_{\alpha}^0, 
  \quad \alpha =1,2,\dots,p .
\end{equation}

We treat the system of equations (\ref{1}) as a single evolutionary equation for the vector 
$\mathbf{u} = \{u_1, u_2, \dots, u_p \}$:
\begin{equation}\label{3}
  \mathbf{B} \frac{d \mathbf{u}}{d t} + \mathbf{A} \mathbf{u} = \mathbf{f}(t),
  \quad 0 < t \leq T,
\end{equation}
where $\mathbf{f} = \{f_1, f_2, \dots, f_p \}$, and for the elements of the 
operator matrices $\mathbf{A}$ and  $\mathbf{B}$ we have the representation
\[
  \mathbf{A} = \{A_{\alpha \beta} \},
  \quad \mathbf{B} = \{B_{\alpha \beta} \} ,
  \quad  \alpha, \beta =1,2,\dots,p .
\]
On the direct sum of spaces \cite{Halmos} 
$\mathbf{H} = H_1 \oplus  H_2 \oplus \cdots \oplus H_p$ we set 
\[
  (\mathbf{u}, \mathbf{v}) = \sum_{\alpha =1}^{p} (u_{\alpha},v_{\alpha})_{\alpha} ,
  \quad \|\mathbf{u} \|^2 = \sum_{\alpha =1}^{p} \|u_{\alpha}\|^2_{\alpha} .
\]
Taking into account (\ref{2}), we define 
\begin{equation}\label{4}
  \mathbf{u}(0) = \mathbf{v}^0,
\end{equation}  
where  $\mathbf{v}^0 = \{v^0_1, v^0_2, \dots, v^0_p \}$.

We consider the Cauchy problem (\ref{3}), (\ref{4}), under the condition
that the operators $\mathbf{A}$  and $\mathbf{B}$ are self-adjoint 
and positive definite in  $\mathbf{H}$:
\begin{equation}\label{5}
  \mathbf{A} = \mathbf{A}^* \geq \delta_A  \mathbf{E},
  \quad \delta_A > 0,
  \quad  \mathbf{B} = \mathbf{B}^* \geq \delta_B  \mathbf{E},
  \quad \delta_B > 0,
\end{equation} 
where $\mathbf{E}$ is the identity operator in $\mathbf{H}$.
The self-adjointness is associated with the fulfillment of
\[
 A_{\alpha \beta} = A^*_{\beta \alpha},
 \quad B_{\alpha \beta} = B^*_{\beta \alpha},
 \quad  \alpha, \beta =1,2,\dots,p 
\] 
for operators of the original system of equations (\ref{1}). 

Here is the simplest a priori estimate for the solution of the Cauchy problem (\ref{3}), (\ref{4}).
We will use it as a guide in investigating the corresponding operator-difference schemes.  
For $\mathbf{D} = \mathbf{D}^*  > 0$ we use the notation $\mathbf{H}_{\mathbf{D}}$ 
for a space $\mathbf{H}$ with the scalar product 
$(\mathbf{y},\mathbf{w})_{\mathbf{D}} = (\mathbf{D}\mathbf{y},\mathbf{w})$ 
and norm $\|\mathbf{y}\|_{\mathbf{D}}= (\mathbf{D}\mathbf{y},\mathbf{y})^{1/2}$.

Scalarly multiplying both sides of equation (\ref{3}) in $\mathbf{H}$ by ${\displaystyle \frac{d \mathbf{u}}{d t}}$,
we obtain 
\[
  \left (\mathbf{B} \frac{d \mathbf{u}}{d t}, \frac{d \mathbf{u}}{d t} \right ) +
  \frac{1}{2} \frac{d }{d t} (\mathbf{A} \mathbf{u}, \mathbf{u}) =
  \left (\mathbf{f},  \frac{d \mathbf{u}}{d t} \right ) .  
\] 
Taking into account (\ref{5}) and 
\[
  \left (\mathbf{f},  \frac{d \mathbf{u}}{d t} \right ) \leq
  \left (\mathbf{B} \frac{d \mathbf{u}}{d t}, \frac{d \mathbf{u}}{d t} \right ) +
  \frac{1}{4} \left (\mathbf{B}^{-1} \mathbf{f}, \mathbf{f} \right ) ,
\] 
we derive the inequality 
\[
  \frac{d }{d t} \| \mathbf{u} \|^2_{\mathbf{A}} \leq 
  \frac{1}{2} \| \mathbf{f} \|^2_{\mathbf{B}^{-1}} .
\] 
We get from it the following a priori estimate 
\begin{equation}\label{6}
  \| \mathbf{u} (t) \|^2_{\mathbf{A}} \leq  \| \mathbf{v}^0 \|^2_{\mathbf{A}}
  + \frac{1}{2} \int_{0}^{t} \| \mathbf{f}(\theta) \|^2_{\mathbf{B}^{-1}} d \theta ,
\end{equation} 
which expresses the stability of the solution of problem (\ref{3}), (\ref{4}) 
with respect to the initial data and right-hand side. 

Systems of evolutionary equations, similar to (\ref{1}),
result from the approximation in space to many applied problems. 
Let us consider here some typical examples do not specifying
the corresponding spaces of continuous and grid functions. 

First, we mention the system of coupled parabolic equations of second order
describing mass transfer in multicomponent media 
\cite{Taylor,Giovangigli}. 
The solution is sought in a bounded domain
$\Omega$,
$u_{\alpha}(\mathbf{x},t)$, $\mathbf{x} \in \Omega$:
\[
  \sum_{\beta =1}^{p} b_{\alpha \beta}(\mathbf{x}) \frac{\partial u_{\beta}}{\partial t} -
  \sum_{\beta =1}^{p} \div (k_{\alpha \beta} (\mathbf{x}) \grad u_{\beta}) +
  \sum_{\beta =1}^{p} r_{\alpha \beta} (\mathbf{x}) u_{\beta} = f_{\alpha}(\mathbf{x},t),
\]
\[
  \mathbf{x} \in \Omega,
  \quad 0 < t \leq T,  
  \quad \alpha = 1, 2, \dots, p. 
\]  
The coefficients $r_{\alpha \beta}$ are associated with the reaction processes whereas 
$k_{\alpha \beta}$ describe the diffusion phenomena
(\emph{main-term} diffusion at $\alpha = \beta$ and \emph{cross-term} diffusion coefficients at $\alpha \neq \beta$). 
For the multicomponent media we have $b_{\alpha \beta} = \delta_{\alpha \beta}b_{\alpha}$, 
where  $\delta_{\alpha \beta}$ is the Kronecker symbol. 

The second example concerns a fluid motion in the porous media. 
The governing equations for a flow in fractured porous media employ the multiple 
porosity model (see, for example,  \cite{Barenblatt,Cheng}).
In this case $u_{\alpha}(\mathbf{x},t)$ is the dynamic pore pressure in an p-porosity model.
For these models it is principal that $b_{\alpha \beta} \neq  \delta_{\alpha \beta}b_{\alpha}$
and $r_{\alpha \beta} \neq 0$ (the Barenblatt model).

\section{Scheme with weights}

To solve approximately operator-differential problem (\ref{3}), (\ref{4}),
we use the standard scheme with weights \cite{Samarskii}.
Introduce a uniform grid in time 
\[
  \overline{\omega}_\tau =
  \omega_\tau\cup \{T\} =
  \{t_n=n\tau,
  \quad n=0,1,...,N,
  \quad \tau N=T\} 
\]
and suppose  $\mathbf{y}^n = \mathbf{y}(t_n), \ t_n = n \tau$.
Let us approximate equation (\ref{3}) via the following two-level difference scheme
\begin{equation}\label{7}
  \mathbf{B} \frac{\mathbf{y}^{n+1} - \mathbf{y}^{n}}{\tau } +
  \mathbf{A} (\sigma \mathbf{y}^{n+1} + (1-\sigma) \mathbf{y}^{n} )  =
  \boldsymbol{\varphi}^{n} ,
\end{equation}
where $\sigma$ is a numerical parameter (weight) within $0 \le \sigma \le 1$
and, for example, $\boldsymbol{\varphi}^{n} = \mathbf{f} (\sigma \mathbf{t}^{n+1} + (1-\sigma) \mathbf{t}^{n} )$.
For simplicity, we restrict ourselves to the case of the equal weight 
for all equations of systems (\ref{1}).
In view of (\ref{4}) we supplement (\ref{7}) with the initial condition 
\begin{equation}\label{8}
  \mathbf{y}^{0} = \mathbf{v}^{0}.
\end{equation}
A detailed study of the schemes  with weights (the necessary and sufficient conditions 
for stability, the choice of a norm) was conducted in \cite{SamGul,SamMatVab}.
Here we restrict ourselves to the simplest estimate of stability for
operator-difference scheme (\ref{7}), (\ref{8}).
Estimate (\ref{6}) will serve us as a guide.

\begin{theorem}\label{t-1}
If $\sigma \geq 1/2$, then operator-difference scheme  
(\ref{7}) is absolutely stable in $\mathbf{H}_{\mathbf{A}}$ and for 
the difference solution the following level-wise estimate is valid:
\begin{equation}\label{9}
  \|\mathbf{y}^{n+1}\|^2_\mathbf{A} \leq  \|\mathbf{y}^{n}\|^2_\mathbf{A} +
  \frac{\tau }{2} \|\boldsymbol{\varphi}^n\|^2_{\left (\mathbf{B} + \left (\sigma - \frac{1}{2} \right ) \tau  \mathbf{A} \right )^{-1}} .
\end{equation}
\end{theorem}

\begin{proof}
We write scheme (\ref{7}) in the form 
\[
  \left (\mathbf{B} + \left (\sigma - \frac{1}{2} \right ) \tau \mathbf{A} \right )
  \frac{\mathbf{y}^{n+1} - \mathbf{y}^{n}}{\tau } +
  \mathbf{A} \frac{\mathbf{y}^{n+1} + \mathbf{y}^{n}}{2}  =
  \boldsymbol{\varphi}^{n} .
\] 
Scalarly  multiplying both sides of this equation in $\mathbf{H}$  by $2(\mathbf{y}^{n+1} - \mathbf{y}^{n})$,
we obtain the equality 
\[
  2 \tau  \left ( \left (\mathbf{B} + \left (\sigma - \frac{1}{2} \right ) \tau \mathbf{A} \right )
  \frac{\mathbf{y}^{n+1} - \mathbf{y}^{n}}{\tau },\frac{\mathbf{y}^{n+1} - \mathbf{y}^{n}}{\tau } \right ) +
\] 
\[
  (\mathbf{A} \mathbf{y}^{n+1}, \mathbf{y}^{n+1}) - (\mathbf{A} \mathbf{y}^{n}, \mathbf{y}^{n}) =
  2 \tau \left (\boldsymbol{\varphi}^{n}, \frac{\mathbf{y}^{n+1} - \mathbf{y}^{n}}{\tau } \right ) .
\] 
Using the inequality 
\[
  \left (\boldsymbol{\varphi}^{n}, \frac{\mathbf{y}^{n+1} - \mathbf{y}^{n}}{\tau } \right ) \leq  
  \left ( \left (\mathbf{B} + \left (\sigma - \frac{1}{2} \right ) \tau \mathbf{A} \right )
  \frac{\mathbf{y}^{n+1} - \mathbf{y}^{n}}{\tau },\frac{\mathbf{y}^{n+1} - \mathbf{y}^{n}}{\tau } \right ) +
\]
\[
  \frac{1}{4} \left (\left (\mathbf{B} + \left (\sigma - \frac{1}{2} \right ) \tau \mathbf{A} \right )^{-1}
  \boldsymbol{\varphi}^{n},\boldsymbol{\varphi}^{n} \right ) ,
\] 
we get the required estimate (\ref{9}).
\end{proof}

Estimate (\ref{9}) is just a grid analog of estimate (\ref{6})
and provides the unconditional stability of the difference scheme 
with weights (\ref{7}), (\ref{8}) under the natural conditions $\sigma \geq 1/2$. 
Considering the corresponding problem for the error, we obtain  convergence of the solution
of operator-difference problem (\ref{7}), (\ref{8}) to the solution of operator-differential 
problem (\ref{3}), (\ref{4})
in $\mathbf{H}_{\mathbf{A}}$ at $\sigma \geq 1/2$ ñ $\mathcal{O}((2 \sigma -1)\tau + \tau^2)$. 
If $\sigma = 1/2$, we have the second-order convergence with respect to $\tau$. 

Operator-difference scheme (\ref{7}) can be written in the canonical form of the
two-level  schemes  \cite{Samarskii} 
\begin{equation}\label{10}
  (\mathbf{B} + \sigma \tau \mathbf{A}) \frac{\mathbf{y}^{n+1} - \mathbf{y}^{n}}{\tau } +
  \mathbf{A} \mathbf{y}^{n} =
  \boldsymbol{\varphi}^{n}.
\end{equation} 
The transition to a new time level requires to solve the problem 
\[
 (\mathbf{B} + \sigma \tau \mathbf{A}) \mathbf{y}^{n+1} = \boldsymbol{\psi}^{n}.
\] 
For the original problem (\ref{1}), (\ref{2}) 
we have to solve the system of coupled equations 
\[
  \sum_{\beta =1}^{p} (B_{\alpha \beta} +
  \sigma \tau A_{\alpha \beta})y_\beta^{n+1} = \psi_\alpha^n,
  \quad \alpha = 1, 2, \dots, p. 
\]  
Various iterative methods can be used for it \cite{SamNikBook,Saad}.

Another opportunity here is to take into account the specific features
of the considered time-dependent problems and to construct splitting schemes, where the transition to a new 
time level is based on the solution of more simpler problems. 
For problems of type (\ref{1}), (\ref{2}) it is natural to employ the splitting schemes,
where the transition to the new time level is performed via solving the problems 
\[
  (B_{\alpha \alpha} +
  \sigma \tau A_{\alpha \alpha })y_\alpha ^{n+1} = \widetilde{\psi}_\alpha^n,
  \quad \alpha = 1, 2, \dots, p.  
\]  
This means that we have to invert only the diagonal part of the operator matrix 
$\mathbf{B} + \sigma \tau \mathbf{A}$ in our computations.

\section{Schemes with a diagonal operator  $\mathbf{B}$}

We start from the case where the problem  of inversion of the operator $\mathbf{B}$ does not exist. 
Such a situation occurs when the operator matrix  $\mathbf{B}$
at the time derivatives is diagonal:
\begin{equation}\label{11}
  B_{\alpha \beta} = \delta_{\alpha \beta}B_{\alpha},
  \quad \alpha = 1,2,\cdots,p .   
\end{equation} 
This class of problems appears in simulation of  mass transfer in multicomponent media. 
In this case the components of the solution vector are coupled due to 
the elements $A_{\alpha \beta}, \ \alpha \neq \beta$ of the operator matrix  $\mathbf{A}$.

Following to \cite{Samarskii,VabLis,SamVabNS,VabVecPap}, let us construct the additive operator-difference schemes 
using the triangular splitting of the operator  $\mathbf{A}$:
\begin{equation}\label{12}
  \mathbf{A} = \mathbf{A}_1 + \mathbf{A}_2,
  \quad  \mathbf{A}_1^* = \mathbf{A}_2 .
\end{equation} 
In the additive representation (\ref{12}) we have
\[
  \mathbf{A}_1 = 
  \begin{pmatrix}
  \frac{1}{2} A_{11} & 0 & \cdots & 0 \\
  A_{21} & \frac{1}{2} A_{22} & \cdots & 0 \\
  \cdots & \cdots & \cdots & 0 \\
  A_{p1} & A_{p2} & \cdots & \frac{1}{2} A_{pp} \\
  \end{pmatrix} ,
  \quad
  \mathbf{A}_2 = 
  \begin{pmatrix}
  \frac{1}{2} A_{11} & A_{12} & \cdots & A_{1p} \\
  0 & \frac{1}{2} A_{22} & \cdots & A_{2p} \\
  \cdots & \cdots & \cdots & A_{p-1p} \\
  0 & 0 & \cdots & \frac{1}{2} A_{pp} \\
  \end{pmatrix}  .
\] 
Instead of (\ref{10}) we use the scheme 
\begin{equation}\label{13}
  \widetilde{\mathbf{B}} \frac{\mathbf{y}^{n+1} - \mathbf{y}^{n}}{\tau } +
  \mathbf{A} \mathbf{y}^{n} =
  \boldsymbol{\varphi}^{n},
\end{equation} 
where the operator $\widetilde{\mathbf{B}}$ has the following factorized form 
\begin{equation}\label{14}
  \widetilde{\mathbf{B}} = 
  (\mathbf{B} + \sigma \tau \mathbf{A}_1) \mathbf{B}^{-1} (\mathbf{B} + \sigma \tau \mathbf{A}_2) .
\end{equation} 
Scheme (\ref{13}), (\ref{14}) is an operator-matrix analog of 
the Samarskii alternating triangle method \cite{Samarskii}.
 
Taking into account (\ref{12}), due to self-adjointness and positive definiteness of 
$\mathbf{B}$ at $\sigma \geq  1/2$ we have 
\[
  \widetilde{\mathbf{B}} = \mathbf{B} + \sigma\tau \mathbf{A} + \sigma^2\tau^2 \mathbf{A}_1 \mathbf{B}^{-1}\mathbf{A}_2,
  \quad   \widetilde{\mathbf{B}} = \widetilde{\mathbf{B}}^* \geq  \mathbf{B} + \sigma\tau \mathbf{A} .
\] 
Similar to Theorem \ref{t-1}, we prove the following statement. 

\begin{theorem}\label{t-2}
If $\sigma \geq 1/2$, then factorized operator-difference scheme  
(\ref{12})--(\ref{14})  is absolutely stable in $\mathbf{H}_{\mathbf{A}}$ and for 
the difference solution the following level-wise estimate is valid:
\[
  \|\mathbf{y}^{n+1}\|^2_\mathbf{A} \leq  \|\mathbf{y}^{n}\|^2_\mathbf{A} +
  \frac{\tau }{2} \|\boldsymbol{\varphi}^n\|^2_{\left (\mathbf{B} + \left (\sigma - \frac{1}{2} \right ) \tau \mathbf{A} +
  \sigma^2\tau^2 \mathbf{A}_1 \mathbf{B}^{-1}\mathbf{A}_2 \right )^{-1}} .
\]
\end{theorem}

The computational implementation of scheme (\ref{12})-(\ref{14}) 
for solving problem (\ref{1}), (\ref{2}), (\ref{11}) 
can be conducted using the sequence of simpler problems:
\[
  \left (B_{\alpha} +
  \sigma \frac{\tau}{2}  A_{\alpha \alpha } \right )y_\alpha ^{n+1/2} = \check{\psi}_\alpha^n,
\] 
\[
  \left (B_{\alpha} +
  \sigma \frac{\tau}{2}  A_{\alpha \alpha } \right )y_\alpha ^{n+1} = \hat{\psi}_\alpha^n,
  \quad \alpha = 1, 2, \dots, p.   
\]
As the scheme with weights (\ref{7}), factorized scheme (\ref{12})-(\ref{14}) has the convergence of second order at  
$\sigma = 1/2$ and of first order at any other weights. 

\section{Generalizations}

If problem (\ref{3}), (\ref{4}) has a non-diagonal operator 
$\mathbf{B}$ ($B_{\alpha \beta} \neq \delta_{\alpha \beta}B_{\alpha}$), then 
additive operator-difference schemes can be constructed using the triangular splitting for both
the operator $\mathbf{A}$ and the operator $\mathbf{B}$. 
Similarly (\ref{12}), we assume that
\begin{equation}\label{15}
  \mathbf{B} = \mathbf{B}_1 + \mathbf{B}_2,
  \quad  \mathbf{B}_1^* = \mathbf{B}_2 .
\end{equation}
Taking into account (\ref{12}), (\ref{15}), we write the scheme with weights (\ref{10}) as follows  
\begin{equation}\label{16}
  \mathbf{C} \frac{\mathbf{y}^{n+1} - \mathbf{y}^{n}}{\tau } +
  \mathbf{A} \mathbf{y}^{n} =
  \boldsymbol{\varphi}^{n},
\end{equation} 
where
\begin{equation}\label{17}
  \mathbf{C} = \mathbf{C}_1 + \mathbf{C}_2,
  \quad \mathbf{C}_1 = \mathbf{B}_1 + \sigma\tau \mathbf{A}_1,
  \quad  \mathbf{C}_2 = \mathbf{B}_2 + \sigma\tau \mathbf{A}_2 .
\end{equation}
In view of (\ref{5}) we have 
\[
  \mathbf{C}_1^* =  \mathbf{C}_2,
  \quad  \mathbf{C}_\alpha \geq \frac{1}{2} (\delta_B +  \sigma\tau \delta_A) \mathbf{E} . 
\] 

The operator $\mathbf{C}_1 + \mathbf{C}_2$  can be represented like this 
\[
  \mathbf{C}_1 + \mathbf{C}_2 = 
  \frac{1}{2\varepsilon} (\mathbf{C}_1 + \varepsilon \mathbf{E}) (\mathbf{C}_2 + \varepsilon \mathbf{E})
  - \frac{1}{2\varepsilon}(\mathbf{C}_1 - \varepsilon \mathbf{E}) (\mathbf{C}_2 - \varepsilon \mathbf{E})
\] 
at any  $\varepsilon > 0$. The value of $\varepsilon$ will be defined a little bit later. 
Instead of two-level scheme (\ref{16}) we employ the three-level scheme 
\begin{equation}\label{18}
  \frac{1}{2\varepsilon} (\mathbf{C}_1 + \varepsilon \mathbf{E}) (\mathbf{C}_2 + \varepsilon \mathbf{E})
  \frac{\mathbf{y}^{n+1} - \mathbf{y}^{n}}{\tau } -
\end{equation} 
\[
  \frac{1}{2\varepsilon} (\mathbf{C}_1 - \varepsilon \mathbf{E}) (\mathbf{C}_2 - \varepsilon \mathbf{E})
  \frac{\mathbf{y}^{n} - \mathbf{y}^{n-1}}{\tau } +
  \mathbf{A} \mathbf{y}^{n} =
  \boldsymbol{\varphi}^{n} .
\]
The primary potential advantage of this scheme in compare with scheme (\ref{16}) is that 
its implementation is based on the inversion of the factorized operator 
$(\mathbf{C}_1 + \varepsilon \mathbf{E}) (\mathbf{C}_2 + \varepsilon \mathbf{E})$ at the new time level.

Taking into account that 
\[
  \frac{\mathbf{y}^{n+1} - \mathbf{y}^{n}}{\tau } = 
  \frac{\mathbf{y}^{n+1} - \mathbf{y}^{n-1}}{2\tau } +
  \frac{\tau}{2} \frac{\mathbf{y}^{n+1} - 2 \mathbf{y}^{n} + \mathbf{y}^{n-1}}{\tau^2},
\] 
\[
  \frac{\mathbf{y}^{n} - \mathbf{y}^{n-1}}{\tau } = 
  \frac{\mathbf{y}^{n+1} - \mathbf{y}^{n-1}}{2\tau } -
  \frac{\tau}{2} \frac{\mathbf{y}^{n+1} - 2 \mathbf{y}^{n} + \mathbf{y}^{n-1}}{\tau^2},
\]
we can write scheme (\ref{18}) as follows 
\begin{equation}\label{19}
  \mathbf{C} \frac{\mathbf{y}^{n+1} - \mathbf{y}^{n-1}}{2\tau } +
  \mathbf{D} \frac{\mathbf{y}^{n+1} - 2 \mathbf{y}^{n} + \mathbf{y}^{n-1}}{\tau^2} +
  \mathbf{A} \mathbf{y}^{n} =
  \boldsymbol{\varphi}^{n} ,
\end{equation} 
where 
\[
 \mathbf{D} = \frac{\tau}{2 \varepsilon} (\mathbf{C}_1 \mathbf{C}_2 + \varepsilon^2  \mathbf{E}) .
\] 
In view of $\mathbf{C} = \mathbf{B} + \sigma \tau \mathbf{A}$
we verify directly that operator-difference scheme (\ref{19})  
approximates equation (\ref{3}) with the first-order accuracy with respect to $\tau$ at all $\varepsilon = \mathcal{O}(1)$.
We now formulate the sufficient conditions of stability for this scheme. 
A comprehensive study on the stability of three-level schemes with self-adjoint operators
was done in \cite{Samarskii,SamGul,SamMatVab}. We do not use here the general results on the stability 
of operator-difference schemes from the above works. Similar to \cite{VabBadd}, we obtain immediately  
the simplest estimates of stability with respect to the initial data and right-hand side. 

Taking into account 
\[
  \mathbf{y}^{n} = 
  \frac{1}{4} (\mathbf{y}^{n+1} + 2 \mathbf{y}^{n} + \mathbf{y}^{n-1}) -
  \frac{1}{4} (\mathbf{y}^{n+1} - 2 \mathbf{y}^{n} + \mathbf{y}^{n-1}) 
\] 
we write (\ref{19}) as 
\begin{equation}\label{20}
  \mathbf{C} \frac{\mathbf{y}^{n+1} - \mathbf{y}^{n-1}}{2\tau } +
  \left (\mathbf{D} - \frac{\tau^2}{4} \mathbf{A} \right ) 
  \frac{\mathbf{y}^{n+1} - 2 \mathbf{y}^{n} + \mathbf{y}^{n-1}}{\tau^2} +
\end{equation} 
\[
  \mathbf{A} \frac{\mathbf{y}^{n+1} - 2 \mathbf{y}^{n} + \mathbf{y}^{n-1}}{4} =
  \boldsymbol{\varphi}^{n} .
\]
Introducing 
\[
  \mathbf{v}^{n} = \frac{1}{2} (\mathbf{y}^{n} + \mathbf{y}^{n-1}),
  \quad \mathbf{w}^{n} = \frac{\mathbf{y}^{n} - \mathbf{y}^{n-1}}{\tau} 
\]
we can rewrite (\ref{20}) in the form 
\begin{equation}\label{21}
  \mathbf{C} \frac{\mathbf{w}^{n+1} + \mathbf{w}^{n}}{2} 
  + \mathbf{R}
  \frac{\mathbf{w}^{n+1} - \mathbf{w}^{n}}{\tau } +
  \frac{1 }{2} \mathbf{A}
  ( \mathbf{v}^{n+1} + \mathbf{y}^{n})  =
  \boldsymbol{\varphi}^{n} ,
\end{equation}
where 
\[
  \mathbf{R} = \mathbf{D} - \frac{\tau^2}{4} \mathbf{A} .
\] 
Scalarly multiplying  both sides of (\ref{21}) by  
\[
  2 (\mathbf{v}^{n+1} - \mathbf{v}^{n}) =
  \tau (\mathbf{w}^{n+1} + \mathbf{w}^{n}) ,
\]
we get the equality
\begin{equation}\label{22}
  \frac{\tau}{2} 
  ( \mathbf{C} (\mathbf{w}^{n+1} + \mathbf{w}^{n}),
    \mathbf{w}^{n+1} + \mathbf{w}^{n}) +
  ( \mathbf{R} (\mathbf{w}^{n+1} - \mathbf{w}^{n}),
    \mathbf{w}^{n+1} + \mathbf{w}^{n}) +
\end{equation}
\[
  ( \mathbf{A} (\mathbf{v}^{n+1} + \mathbf{v}^{n}),
    \mathbf{v}^{n+1} - \mathbf{v}^{n}) =
	\tau (\boldsymbol{\varphi}^{n}, \mathbf{w}^{n+1} + \mathbf{w}^{n} ) .
\]
For the right-hand side we use the estimate
\[
  (\boldsymbol{\varphi}^{n}, \mathbf{w}^{n+1} + \mathbf{w}^{n} ) \leq 
  \frac{1 }{2} 
  ( \mathbf{C} (\mathbf{w}^{n+1} + \mathbf{w}^{n}) +
  \frac{1}{2} 
  (\mathbf{C}^{-1} \boldsymbol{\varphi}^n, \boldsymbol{\varphi}^n) .
\]
This yields from (\ref{22}) the inequality 
\begin{equation}\label{23}
  \mathcal{E}_{n+1} \leq 
  \mathcal{E}_{n} +
  \frac{\tau}{2} 
  (\mathbf{C}^{-1} \boldsymbol{\varphi}^n, \boldsymbol{\varphi}^n) ,
\end{equation}
where we use the notation 
\[
  \mathcal{E}_{n} = 
  ( \mathbf{A} \mathbf{v}^{n}, \mathbf{v}^{n})
  +   ( \mathbf{R} \mathbf{w}^{n}, \mathbf{w}^{n}) .
\]
Inequality (\ref{23}) will be the desired a priori estimate, 
if we show that $\mathcal{E}_{n}$ defines the squared norm of the difference solution.
Due to the positivity of $\mathbf{A}$, it is sufficient to require the 
non-negativity  of the operator  $\mathbf{R}$.

With the above-mentioned notations we have 
\[
  \mathbf{R} =
  \frac{\tau}{4\varepsilon} (\mathbf{C}_1 + \varepsilon \mathbf{E}) (\mathbf{C}_2 + \varepsilon \mathbf{E}) + 
  \frac{\tau}{4\varepsilon} (\mathbf{C}_1 - \varepsilon \mathbf{E}) (\mathbf{C}_2 - \varepsilon \mathbf{E}) -
  \frac{\tau^2}{4} \mathbf{A} \geq 
\] 
\[
  \frac{\tau}{4\varepsilon} (\mathbf{C}_1 + \varepsilon \mathbf{E}) (\mathbf{C}_2 + \varepsilon \mathbf{E}) -
  \frac{\tau^2}{4} \mathbf{A} >  \frac{\tau^2}{4} (\sigma - 1) \mathbf{A} .
\] 
Thus, $\mathbf{R} > 0$ at $\sigma \geq 1$.
The result of our considerations is the following statement.

\begin{theorem}\label{t-3}
If $\sigma \geq 1$, then operator-difference scheme (\ref{12}), (\ref{15}), (\ref{18})  is
absolutely stable, and for the difference solution a priori estimate (\ref{23}) holds with
\[
  \mathcal{E}_{n} = 
  \left \| \frac{\mathbf{y}^{n} + \mathbf{y}^{n-1}}{2} \right \|^2_{\mathbf{A}} +
  \left \| \frac{\mathbf{y}^{n} - \mathbf{y}^{n-1}}{\tau} \right \|^2_{\mathbf{R}} .
\]
\end{theorem}

The proven estimate (\ref{23}) ensures the stability of operator-difference scheme
(\ref{12}), (\ref{15}), (\ref{18}) with respect to the initial data and right-hand side.
It can be treated as a more complex analog of estimate (\ref{9}) 
and it agrees with  estimate (\ref{6}) for the solution of the problem
(\ref{3}), (\ref{4}).

The computational implementation of scheme (\ref{18}) involves the solution of the grid problem
\begin{equation}\label{24}
  (\mathbf{C}_1 + \varepsilon \mathbf{E}) (\mathbf{C}_2 + \varepsilon \mathbf{E}) \mathbf{y}^{n+1} 
  = \boldsymbol{\psi}^n
\end{equation} 
at the new time level.
Introducing the auxiliary unknown $\mathbf{y}^{n+1/2}$, we have for (\ref{24}) the following representation
\[
 (\mathbf{C}_1 + \varepsilon \mathbf{E}) \mathbf{y}^{n+1/2} = \boldsymbol{\psi}^n,
\] 
\[
 (\mathbf{C}_2 + \varepsilon \mathbf{E}) \mathbf{y}^{n+1} = \mathbf{y}^{n+1/2} .
\] 
Taking into account the triangular structure of the operators $(\mathbf{C}_1$ and $(\mathbf{C}_2$, we
sequentially solve the problem
\[
  (B_{\alpha\alpha} + \sigma \tau  A_{\alpha \alpha } + 2\varepsilon E_\alpha  )y_\alpha ^{n+1/2} = \check{\psi}_\alpha^n,
\] 
\[
  (B_{\alpha\alpha} + \sigma \tau  A_{\alpha \alpha } + 2\varepsilon E_\alpha  )y_\alpha ^{n+1} = \hat{\psi}_\alpha^n,
  \quad \alpha = 1, 2, \dots, p,   
\]
where $E_\alpha$  is the identity operator in $H_\alpha$.

\bibliographystyle{amsplain}
\bibliography{BadditSys}

\end{document}